\documentclass[letterpaper, 10 pt, conference]{ieeeconf}  

\IEEEoverridecommandlockouts                              

\usepackage{xcolor,graphicx,graphics,epsfig,float}
\usepackage{amsfonts,amsmath,amssymb,bm,bbm}
\usepackage{hyperref}
\usepackage{fancyhdr}
\usepackage{subfigure}
\usepackage{tikz}
\usepackage{lineno}  

\newtheorem{Theorem}{Theorem}

\newtheorem{Proposition}[Theorem]{Proposition}

\date{}

\title{\LARGE \bf
Optimal motion of a scallop: some case studies}

\author{Rosario Maggistro$^{1}$ and Marta Zoppello$^{2}$
\thanks{$^{1}$Department of Management, Universit\`{a} Ca' Foscari Venezia,
Fondamenta S. Giobbe, 873, 30121 Cannaregio, Venezia, Italy
        {\tt\small rosario.maggistro@unive.it}}%
\thanks{$^{2}$Department of Computer Science, Universit\`{a} di Verona, Strada le Grazie, 15, 37134, Verona, Italy
        {\tt\small marta.zoppello@univr.it}}%
}

\begin{document}

\maketitle
\thispagestyle{empty}
\pagestyle{empty}

\begin{abstract}
In this paper we focus on a two-link swimmer called scallop which moves changing dynamics between two fluids regimes. We address and solve explicitly two optimal control problems, the minimum time one and the minimum quadratic cost needed to move the swimmer between two fixed positions using a periodic control. Considering only one switching in the dynamics and exploiting the structure of the equation of motion we are able to split the problem into simpler ones. We solve explicitly each sub-problem obtaining a discontinuous global solution. Then we approximate it  through a suitable sequence of continuous functions. Finally, we show numerical simulations suggesting that to switch less times is the best strategy for both costs.
\end{abstract}
\begin{keywords}
	Scallop swimmer, switching dynamics, optimal control, Pontryagin's Maximum Principle.
\end{keywords}
\section{INTRODUCTION}
The study of locomotion strategies in fluids and its link with the construction of artificial devices that can self-propel in fluids have generated, in the recent years, considerable interest by different research communities. Theories of swimming generally utilize either low Reynolds
number approximation \cite{AlougesDeSimone13, AlougesDeSImone08} and Resistive Force Theory (RTF) \cite{GrayHancock55}, or the assumption of
inviscid ideal fluid dynamics (high Reynolds number) \cite{ChambrionMunnier,CardinZoppello}.
These two different regimes are also distinct in terms
of the mechanism of locomotion \cite{Childress}--\cite{Lighthill}.
In this paper we focus on swimmers immersed in
these two kind of fluids which produce a linear
dynamics. In particular we consider the system describing
the motion of a scallop 
for which several strategies were proposed in order to overcome the so called Scallop theorem. According to it a swimmer that moves opening and closing periodically its valves, cannot achieve any net motion both in a viscous and ideal irrotational fluid, because of the time reversibility of the equations \cite{Purcell,AlougesDeSImone08}.
Some authors in order to solve the no net motion problem,
added, for example, a degree of freedom
introducing the Purcell swimmer \cite{Purcell}, or the three
sphere swimmer \cite{Golestain}. Others, instead, supposed the
scallop immersed in a non Newtonian fluid, in which
the viscosity is not constant, ending up with a non
reversible dynamics \cite{Cheng}--\cite{Qiu}.
Here, inspired by an idea of A. Bressan,
we use the approach proposed in \cite{Meccanica} where a change in the fluid's regime between the opening and closing of the valves is supposed. Moreover an hysteresis operator is introduced 
through a thermostat, see Fig. \ref{termostato} (see \cite{Visintin} for mathematical models for hysteresis), to model a delay
in this change of fluid's regime. This operator has been used in other contexts to model the switching dynamics phenomenon 
\cite{MaggistroBagagiolo}. In our framework the use of this operator on the one hand permits both to overcome the Scallop theorem and make the system completely controllable. On the other hand it requires the use of a continuous periodic control (angular velocity) whose explicit expression is difficult to derive. In this paper we address the optimal time and optimal quadratic cost problem. This has been done for other kind of swimmers, for example the Purcell three link one \cite{GiraldiMartinon13,GiraldiZoppello14}. 



 Our main result is to provide an explicit periodic solution of two optimal control problems, i.e. the minimum time and the linear quadratic optimal control one, so that the scallop can move between two fixed positions.
The existence of such a solution is guaranteed thanks to the controllability of the scallop system with switching dynamics and only relaxing the hypothesis on the regularity of the control. Moreover, due to the structure of the equations of motion, we reduce to the case of only one switching and  overcome the difficulty of managing the changing dynamics splitting the optimal control problem into simpler sub-problems which, that can be solved in the relaxed framework. Thus using the Pontryagin Maximum Principle (PMP)  \cite{Agrachev,Trelat}, we solve explicitly the relaxed sub-problems.
Then we recover a continuous periodic solution building a suitable sequence of continuous controls which approximate in $L^1$ the (possible) discontinuous optimal one. In the quadratic case where only the square of the $L^2$-norm of the control is take into account as a cost function, we prove that the control that minimize that cost is the same one that solve the minimum time problem.\\Moreover, through numerical simulations we give some hints on what happens in the case of $n$ switchings in the dynamics. More precisely we find an upper bound for the cost with $n$ switchings  and show numerical simulations that suggest that performing less switchings is the best strategy.

\section{Setting of the problem}
\subsection{The model}
We recall the scallop swimmer introduced in \cite{Meccanica} and present its equation of motion as a system of ordinary differential equations. The swimmer is modelled as an articulated rigid body immersed in a fluid that changes its configuration.
It is composed by two rigid valves of elliptical shape of mass $m$, major semiaxes $a$ and minor $b$ with $b<<a$. The valves are joined in order that they can be opened and closed. Moreover this body is constrained to move along the $\vec{e}_x$-axis and is symmetric with respect to it. Finally we will neglect the interaction between the two valves. 
In order to determine completely its state we need the position of the juncture point $x$ and the orientation of each valve $\theta$ with respect to the $\vec{e}_x$-axis. 
\begin{figure}[H]
	\centering
	\includegraphics[scale=0.20]{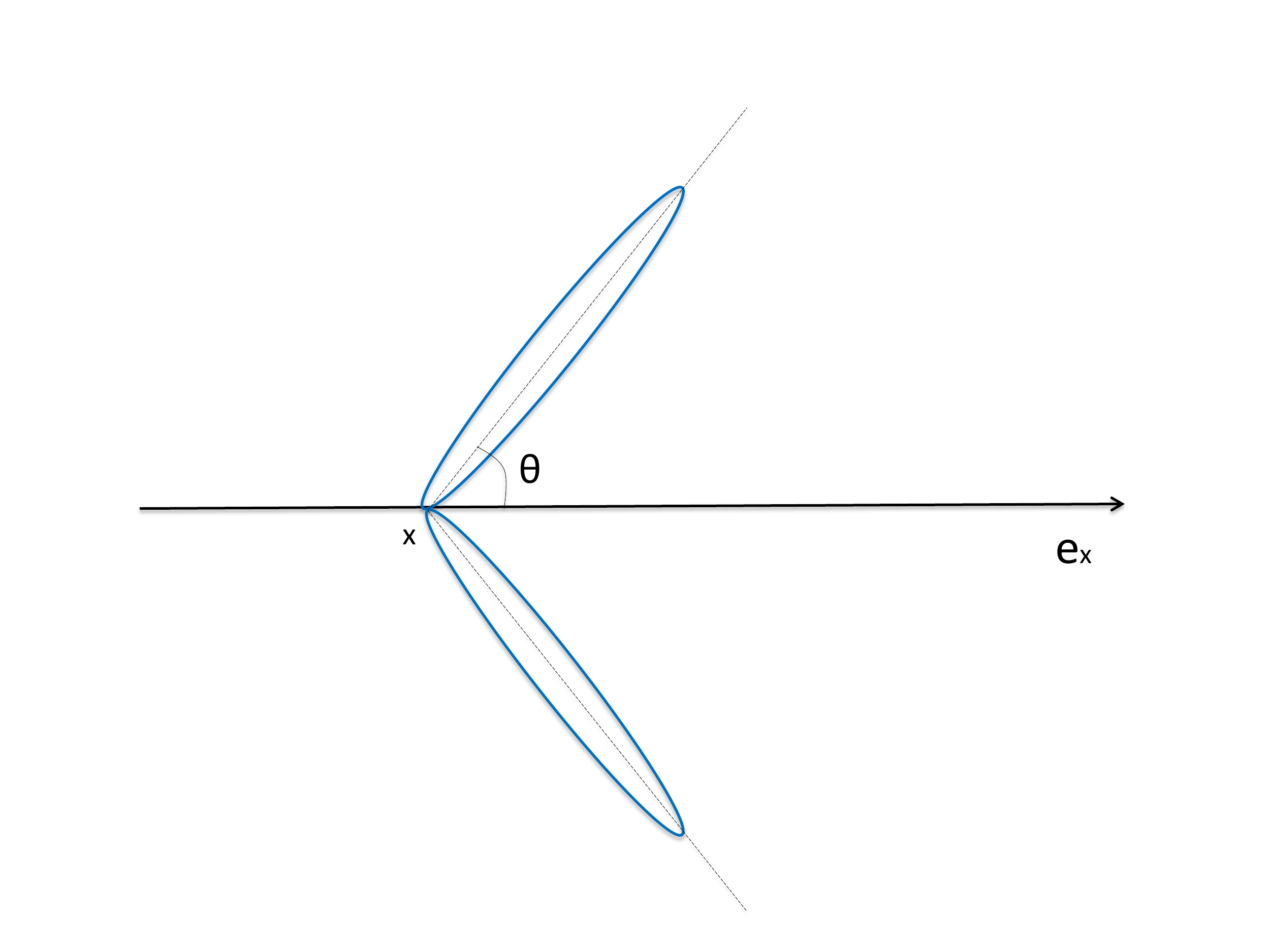}
	\caption{The scallop configuration}
	\label{fig:Scallop}       
\end{figure}
The temporal evolution of these coordinates is obtained solving the Newton's equations 
coupled with the Navier-Stokes equations relative to the surrounding fluid. We will face this problem considering the body as immersed in two kinds of different fluids: one viscous at low Reynolds number in which we neglect the effects of inertia, 
and another one ideal inviscid and irrotational, in which we neglect the viscous forces in the Navier-Stokes equations. 
In the viscous case we consider the two valves of the scallop as one dimensional links and use the local drag approximation of the Resistive Force Theory \cite{GrayHancock55} to compute the total drag force exerted by the fluid on the swimmer. According to it the density of the force $\vec{f}_{i}$ acting on the $i$-th segment at the point of arc-length $s$, is assumed to depend linearly on the velocity of that point. It is defined by
\begin{equation}
\vec{f}_i(s) :=-\xi \left( \vec{v}_i(s) \cdot \vec{e}_i \right) \vec{e}_i- \eta \left(\vec{v}_i(s) \cdot \vec{e}_i^{\bot}\right) \vec{e}_i^{\bot},
\label{ForceByResistiveTheory}
\end{equation}
where $\xi$ and $\eta$ are respectively the drag coefficients in the directions parallel, $\vec{e}_i$, and perpendicular, $\vec{e}_i^{\bot}$, to each link. 
 Integrating the density of force and neglecting inertia, (sum of forces equal zero), we obtain the following dynamics (see \cite{Meccanica} for details)
\begin{equation}
\label{x_viscous}
\dot{x}=V_1(\theta)\dot{\theta}=\frac{a \eta\sin(\theta)}{\xi\cos^2(\theta)+\eta\sin^2(\theta)}\dot{\theta}
\end{equation}

In the ideal case, instead, it has been proven in \cite{Bressan07}, \cite{CardinZoppello} that the system body $+$ fluid is geodetic with Lagrangian given by the sum of the kinetic energy of the body ($T^b$) and the one of the fluid ($T^f$).
$$
T^{tot}=T^b+T^f
$$
Moreover since inertial forces dominates over the viscous ones, in order to derive the kinetic energy of the fluid  we will make use of the concept of \textit{added mass} \cite{Bessel28}. To this end we consider each valve as a thin ellipse. Following a procedure introduced by Alberto Bressan in \cite{Bressan07}, in order to end up with a control system we perform a partial Legendre transformation on the kinetic energy  defining
$$p=\frac{\partial T^{tot}}{\partial\dot{x}}$$
from which we derive the equation of motion for $x$
\begin{equation}
\dot{x}=V_2(\theta)\dot{\theta}=\frac{2 a\dot{\theta}\sin\theta(m+m_{22})}{2(m+m_{11}\cos^2\theta+m_{22}\sin^2\theta)}
\end{equation}
where $m_{ii}$, $i=1\cdots3$, are the added mass coefficients of the ellipse \cite{Bessel28}.\\
Finally, to set notation, from now on we name $F_i$ the primitives of the dynamics $V_i$ $i=1,2$.
\subsection{Thermostatic-like case and controllability}
Our idea is now to use the angular velocity of opening and closing of the valves as a control, and according to \cite{Meccanica}, it is possible to overcome the famous \textit{Scallop theorem}, swithching between the two fluid regimes. More precisely if the valves are opening ($\dot\theta>0$) we suppose that the scallop is immersed in an ideal fluid; instead when the valves are closing  ($\dot\theta<0$) we assume the scallop immersed in a viscous fluid. This idea is inspired by \cite{Cheng} where the fluid has a pseudoelastic nature that helps the valve opening but resist the valve closing. 
Furthermore in \cite{Meccanica} it has been proven that the system is completely controllable using a periodic control $\dot\theta$ only introducing a delayed thermostatic rule. We suppose that the dynamics $V$  depends on the angle  $\theta \in ]0, \frac{\pi}{2}[$, and also depends on a discrete variable $w \in \left\{1, 2\right\}$, which characterizes the fluid regime ($2$ ideal or $1$ viscous) whose evolution is governed by a delayed thermostatic rule, $h_{\varepsilon}[\cdot]$, subject to the evolution of the control $u$. In Fig. \ref{termostato} the behavior of such a rule  is explained, correspondingly to the choice of a fixed threshold parameter $\varepsilon>0$ (see \cite{Visintin} for a complete explanation).
This rule models a situation in which the change between the two fluid regimes is not abrupt from opening to closure of the valves, but the system remains in the fluid regime of origin as long as the angular velocity has reached the threshold $\varepsilon$, then switches.
\begin{figure}[H]
		\begin{flushleft}
			\includegraphics[scale=0.22]{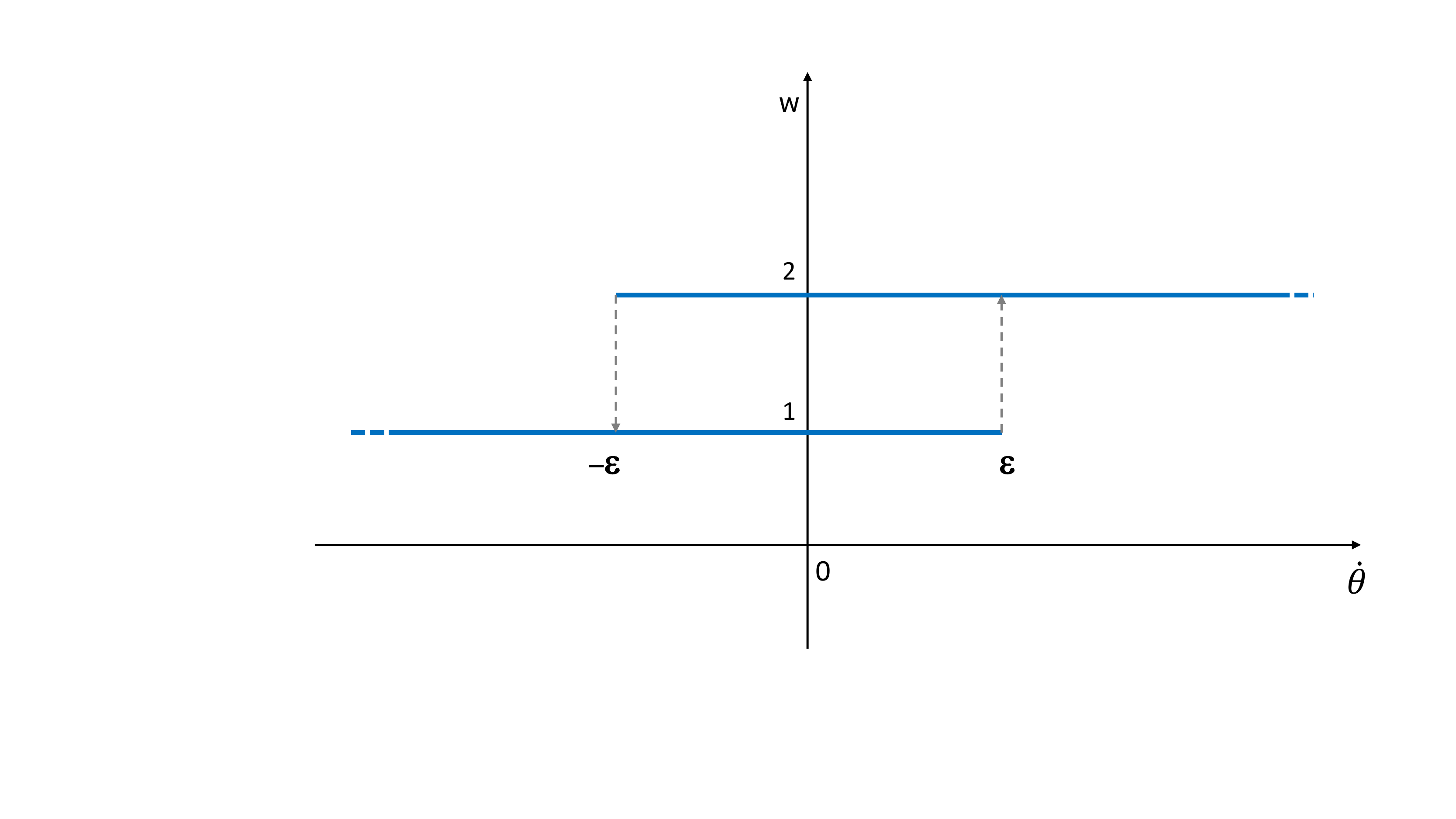} 
			\vspace{-1cm}
\caption{\label{termostato} The thermostatic approximation}
\end{flushleft}
\end{figure} 
The controlled evolution is then given by
\begin{equation}
\label{control_syst_scallop_hyst}
\begin{cases}
\dot{x}(t) = V_{w(t)}(\theta(t))u(t), \\
\dot\theta(t)=u(t)\\
w(t)= h_{\varepsilon}[u](t) \\
x(0)=x_0, \ \ \theta(0)=\theta_0\\
u(0)=u_0, \ \ w(0)= w_0
\end{cases} 
\end{equation}
\section{Minimum time optimal control problem}
\label{Sec:Minimum_time}
In this section we study a minimum time problem for the scallop, we provide the conditions for which there exists an optimal solution and compute it explicitly.\\
We assume that the swimmer starts at the initial
configuration $(x(0), \theta(0), u(0))$ and we want to
find a swimming strategy that minimizes the time to reach
the final configuration $(x(t_f), \theta(t_f), u(t_f))$. Note that in \cite{Meccanica} we consider three possible cases: i) $u_0 \in [-\varepsilon, \varepsilon]\,, w_0=2$; ii) $u_0<-\varepsilon \,, w_0=~1$; iii) $u_0>\varepsilon \,, w_0=2$. Here, we will consider only the case i) observing that the other two cases can be treated in analogous way. 
Then we have
\begin{equation}\label{Initial}
\begin{cases}
\displaystyle\inf_u \int_0^{t_f} ds\\
\dot{x}(t)=V_w(\theta(t))u(t), \qquad \forall t \in [0,t_f], \\
\dot{\theta}(t)=u(t), \qquad  \forall t \in [0,t_f], \\
w(t)= h_{\varepsilon}[u](t) \\
u\in C^0[0,t_f], \qquad u(0)=u(t_f)=u_0,\\
u(t) \in U:= [-\varepsilon, \varepsilon],\qquad \forall t \in [0,t_f], \\
x(0)=x_0, \qquad x(t_f)=x_f,\\
\theta(0)=\theta(t_f)=\theta_0

\end{cases}
\end{equation}
It has been already proved in \cite{Meccanica} that the system \eqref{Initial} is controllable by using a control $u$ in $C^0[0,t_f]$. However, 
since $u\in C^0[0,t_f]$, the theorem of Filippov-Cesary \cite{Trelat} can not be applied. Indeed it is based on approximating the optimal solution with a minimizing sequence that should converge exactly to the optimal solution. In the hypotheses of Filippov-Cesary theorem the space in which the controls live is $L^{\infty}(\cdot)$, and the dynamics evaluated in the sequence of minimizing controls converges weakly-star in $L^{\infty}(\cdot)$ to the dynamics evaluated in the optimal control. This is due to the fact that the space of admissible controls is complete with respect to the $L^1(\cdot)$ topology. In our case instead we have that the space of admissible controls is $C^0[0,t_f]$ which is not complete with respect to the $L^1(0,t_f)$ topology. Therefore, the theorem can not be used to prove the existence of the optimal solution of \eqref{Initial}. Actually we have the following
\begin{Proposition}\label{NoExist}
	There is no solution of the minimum time optimal control problem \eqref{Initial}.
\end{Proposition}
\begin{proof}
Suppose for absurd that there exists an optimal solution $(x^*, \theta^*, u^*) \in C^0[0,t_f]\times C^1[0,t_f]\times C^0[0,t_f]$ of the problem \eqref{Initial} with $x^*=x(u^*)$ and $\theta^*=\theta(u^*)$. This means that there exists a minimizing trajectories sequence $(x_n, \theta_n)_n \in C^0[0,t_f]\times C^1[0,t_f]$ (with $x_n=x(u_n)$ and $\theta_n=\theta(u_n)$) such that
\begin{equation}\label{limitraject} \lim_{n\to \infty} \Vert(x_n, \theta_n)-(x^*, \theta^*)\Vert_\infty=0.
\end{equation} 
Condition \eqref{limitraject} implies that $u_n\to u^*$ in $L^1(0, t_f)$ from which follows that $u^*$ not necessarily belongs to $C^0[0,t_f]$, since $C^0[0,t_f]$ is not complete with respect to the $L^1(0, t_f)$ topology. This fact contradicts the hypothesis, hence the thesis holds.
\end{proof} 
As consequnce of the Proposition \ref{NoExist}, to obtain an optimal solution for \eqref{Initial} our strategy consists in relaxing the hypothesis on the regularity of the control, to prove the existence of a solution and then try to approximate it with sequences of continuous functions which converge to the relaxed solution.
The relaxed formulation of \eqref{Initial} is	
\begin{equation}\label{Initial2}
\begin{cases}
\displaystyle\inf_u \int_0^{t_f} ds\\
\dot{x}(t)=V_w(\theta(t))u(t), \qquad \forall t \in [0,t_f], \\
\dot{\theta}(t)=u(t), \qquad  \forall t \in [0,t_f], \\
u\in L^{\infty}[0,t_f], \  
u(t) \in U:= [-\varepsilon, \varepsilon],\ \forall t \in [0,t_f], \\
w(t)= h_{\varepsilon}[u](t) \\
x(0)=x_0, \qquad x(t_f)=x_f,\\
\theta(0)=\theta(t_f)=\theta_0\\
\end{cases}
\end{equation}
Note that in \eqref{Initial2}, since the control $u$ is in $L^{\infty}$, we do not consider the initial and final condition on the control $u$ because now we are interesting on its value on positive measure intervals instead of on some points.  Moreover, even if the control is no more continuous the system remains controllable (as it is shown in the final numerical simulations in \cite{Meccanica}) since  $C^0[0,t_f]\subset L^{\infty}[0,t_f]$. Now by applying Filippov-Cesary Theorem (\cite{Trelat}), there exists a minimal time such that the constraints are satisfied i.e., the infimum in \eqref{Initial2} is a minimum.  \\
Now, the aim is to find an explicit solution of \eqref{Initial2}. At first we suppose that in order to arrive in $x_f$ starting from $x_0$ the dynamics switches from $2$ to $1$ only one time as in case i) in \cite{Meccanica}. Therefore we can split \eqref{Initial2} in two simpler sub-problems.
Indeed, fixed $\Delta x=F_2(\theta(t_1))-F_1(\theta(t_1))$ 
according to \cite{Meccanica} there exists $\theta_1$ such that $\theta(t_1)=\theta_1$, with $t_1$ the first switching instant, such that the corresponding $\Delta x$ is exactly $x_f-x_0$. Then we have
\begin{equation*}
(P1)\begin{cases}
\displaystyle\min_u t_1\\
\dot{\theta}(t)=u(t), \qquad \forall t \in [0, t_1],\\
u(t) \in [-\varepsilon, \varepsilon], \qquad \forall t \in [0, t_1],\\
\theta(0)=\theta_0 \quad \theta(t_1)=\theta_1
\end{cases}
\end{equation*}
\begin{equation*}
(P2) \begin{cases}
\displaystyle\min_u (t_f-t_1)\\
\dot{\theta}(t)=u(t), \qquad \forall t \in [0, t_f-t_1],\\
u(t) \in [-\varepsilon, \varepsilon], \quad \forall t \in [0, t_f-t_1],\\
\theta(0)=\theta_1 \quad \theta(t_f-t_1)=\theta_0.\\
\end{cases}
\end{equation*}	

We start studying the problem $(P1)$. We apply the Pontryagin Maximum Principle (PMP) to the associated Hamiltonian
\begin{equation}\label{H1}
H(\theta, p)=1+pu,	
\end{equation}
and we get
\begin{equation}\label{PMP}
\dot{p}=-\frac{\partial H}{\partial \theta}=0 \Longrightarrow p(t)=p_0. 
\end{equation}
By the stationarity condition $H(\theta(t_1), p(t_1))=1+p_0u(t_1)$ with $\theta(t_1)=\theta_1$ we get\\
\textbf{Case 1:} if $\theta_1< \theta_0$ then $ u(t_1)=-\varepsilon$ and $p_0=1/\varepsilon$;\\
\textbf{Case 2:} if $\theta_1> \theta_0$ then $u(t_1)=\varepsilon$ and $p_0=-1/\varepsilon$.\\
Accordingly we have
\begin{equation}\label{control1}
\small
u=
\begin{cases}
\varepsilon & \text{if} \quad p_0<0\\
-\varepsilon & \text{if} \quad p_0>0,
\end{cases}\,, \ \theta(t)=
\begin{cases}
\varepsilon t+\theta_0 & \text{if} \quad p_0<0, \\
-\varepsilon t +\theta_0  &\text{if} \quad p_0>0
\end{cases}
\end{equation}

and
\begin{equation}\label{optimalt1}
t_1=\begin{cases}
(\theta_0-\theta_1)/\varepsilon  & \text{if} \quad \theta_1< \theta_0,\\
(\theta_1-\theta_0)/\varepsilon & \text{if} \quad\theta_1> \theta_0.
\end{cases}
\end{equation}
%
Now we consider the problem $(P2)$ and let $t_2=t_f-t_1$. Proceeding as in \eqref{H1}-\eqref{PMP}  
and by the stationarity condition $H(\theta(t_2), p(t_2))=1+p_0u(t_2)$ with $\theta(t_2)=\theta_0$ follows that\\
\textbf{Case 1:} if $\theta_1< \theta_0$ then $ u(t_2)=\varepsilon$ and $p_0=-1/\varepsilon$;\\
\textbf{Case 2:} if $\theta_1> \theta_0$ then $u(t_2)=-\varepsilon$ and $p_0=1/\varepsilon$.\\ Hence
\begin{equation}\label{control2}
u=
\begin{cases}
\varepsilon & \text{if} \quad p_0<0\\
-\varepsilon & \text{if} \quad p_0>0,
\end{cases}\,, \ \theta(t)=
\begin{cases}
\varepsilon t+\theta_1 & \text{if} \quad p_0<0, \\
-\varepsilon t +\theta_1  &\text{if} \quad p_0>0
\end{cases}
\end{equation}
and
\begin{equation*}
t_2=\begin{cases}
(\theta_0-\theta_1)/\varepsilon & \text{if } \theta_1< \theta_0, \\
(\theta_1-\theta_0)/\varepsilon  & \text{if } \theta_1> \theta_0.
\end{cases}
\end{equation*}
Therefore we get in both cases
\begin{equation}
t_f=t_1+t_2=\frac{2|\theta_0-\theta_1|}{\varepsilon}
\end{equation}

We have obtained the bang bang optimal control \eqref{control1}-\eqref{control2} supposing that we do not need a continuous control, indeed the existence of a continuous optimal control is not ensured. Nevertheless we can use a continuous approximation of that bang bang control using piecewise linear functions as shown in Fig. \ref{control_approx} and computed below.
\begin{figure}[H]
\centering
\includegraphics [width=0.25\textwidth]{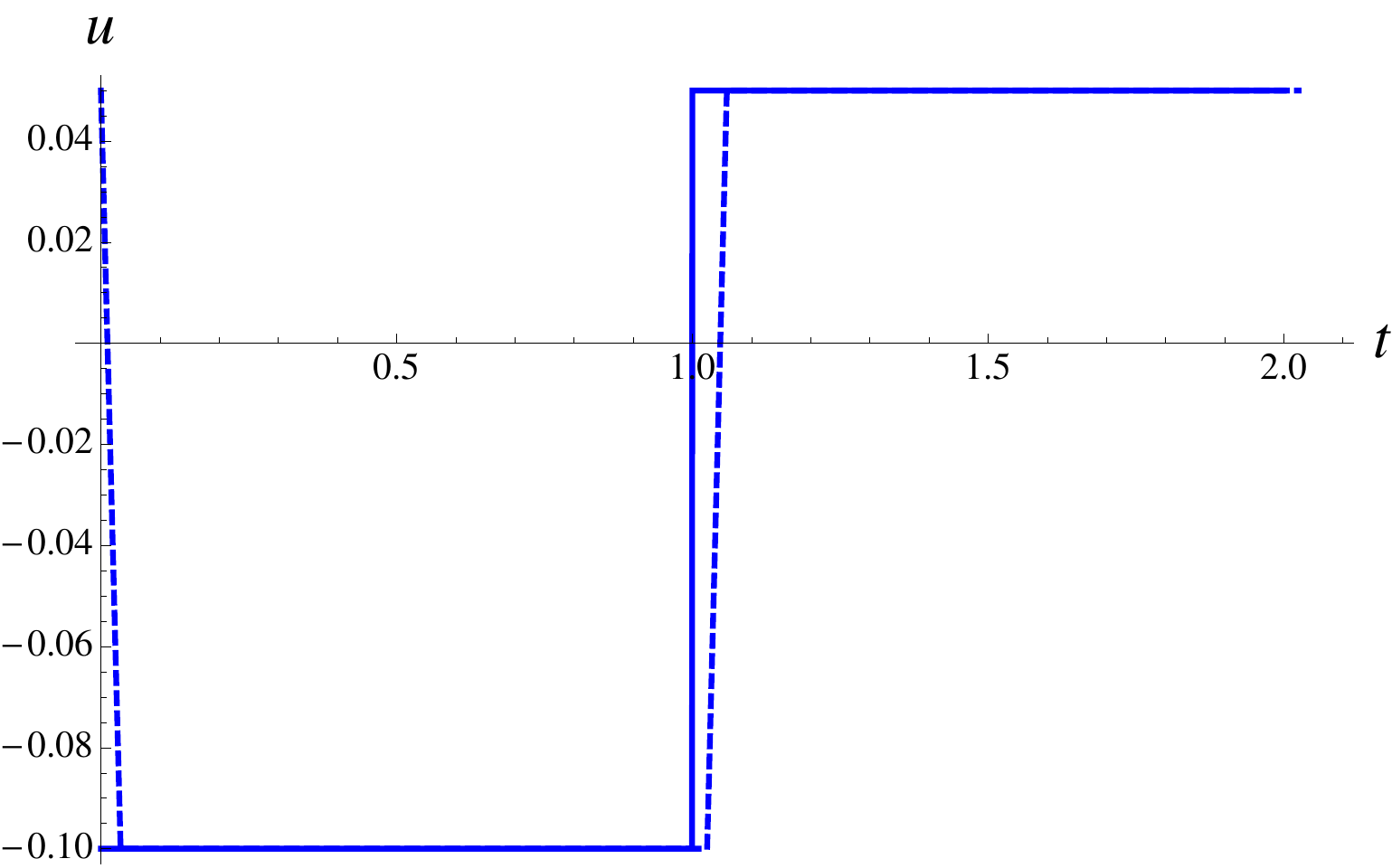}
\caption{\label{control_approx} Continuous approximation of the control in \eqref{control1}-\eqref{control2}}
\end{figure}
In the situation described by \textbf{Case 1} for the problem $(P1)$ let
\begin{equation}\label{u_k}
u_k(t)=
\begin{cases}
-kt(\varepsilon+u_0)+u_0 & 0\leq t\leq 1/k,\\
-\varepsilon & 1/k\leq t\leq \tilde{t}_{1_k},
\end{cases}
\end{equation}
\begin{equation}\label{theta_k}
\theta_k(t)=
\begin{cases}
\displaystyle\frac{-kt^2}{2}(\varepsilon + u_0)+u_0t+\theta_0 & 0\leq t\leq 1/k,\\
\displaystyle-\varepsilon(t-1/k)+\frac{u_0-\varepsilon}{2k}+\theta_0 & 1/k\leq t\leq \tilde{t}_{1_k}.
\end{cases}
\end{equation}
be the continuous approximation of the control and the corresponding trajectory in \eqref{control1} . Note that now we have to consider again the fact that in \eqref{Initial} $u(0)=u_0=u(t_f)$. Therefore we can compute the new switching time $\tilde{t}_{1_k}$ imposing the following
\begin{equation*}
\displaystyle\theta_k(\tilde{t}_{1_k})=\theta_1 \Longleftrightarrow \tilde{t}_{1_k}=\frac{u_0+\varepsilon+2k(\theta_0-\theta_1)}{2k\varepsilon}.
\end{equation*}
It is clearly different from the optimal switching time $t_1$ \eqref{optimalt1} but converges to it as $k\to\infty$.\\
In \textbf{Case 2} for $(P1)$ and  \textbf{Case 1} and \textbf{2}  for $(P2)$, similar approximations to ones in \eqref{u_k}--\eqref{theta_k} hold.\\
Therefore we compute the time $\tilde{t}_{2_k}$ imposing the final condition for $\theta_k$
\begin{equation*}
\displaystyle\theta_k(\tilde{t}_{2_k})=\theta_0 \Longleftrightarrow \tilde{t}_{2_k}=\frac{\varepsilon+k(\theta_0-\theta_1)}{k\varepsilon},
\end{equation*}
Finally, putting together the values of $\tilde{t}_{1_k}$ and $\tilde{t}_{2_k}$ we get the final time 
\begin{equation} \label{t_k}
\displaystyle t_{f_k}=\tilde{t}_{1_k}+\tilde{t}_{2_k}=\frac{u_0+3\varepsilon+4k(\theta_0-\theta_1)}{2k\varepsilon}
\end{equation}
In both cases above, the approximation through piecewise linear function implies that:
\begin{equation*}
\begin{aligned}
& \lim_{k\to +\infty}u_k=u \ \text{in} \ L^1([0, t_f])\\
& \lim_{k\to +\infty}t_{f_k}=t_f.
\end{aligned}
\end{equation*}
Moreover we have the following proposition.
\begin{Proposition}\label{convtraiett}
	The trajectory $\theta_k$ converges uniformly to $\theta$, i.e. 
	$$
	\lim_{k\to +\infty} \theta_k = \theta \ \text{in} \ L^{\infty}([0, t_f])
	$$
\end{Proposition}
\begin{proof}
We have just shown that $ \lim_{k\to +\infty} u_k =~u$ in $L^{1}([0, t_f])$ and $u_k$ are equibounded in $L^{\infty}([0, t_f])$. 
Thus we have the following estimate
$$
||\theta_k-\theta||_{\infty}\leq\int_0^{t_{f_k}} |u_k(s)-u(s)|\,ds
$$
where we assume to continuously extend the control $u$ in the interval $[t_f, t_{f_k}]$.
Therefore we conclude for the $L^1$ convergence of $u_k$ to $u$. 
\end{proof}

\section{Linear Quadratic optimal control problem}
In this section we minimize a linear quadratic cost instead of the minimum time as in the previous section. In particular, we consider the following problem
\begin{equation}\label{quadratic}
\begin{cases}
\displaystyle\inf_{u \in [-\varepsilon, \varepsilon]} \int_0^{t_f} A u^2(s)+B \theta(s)^2\, ds, \quad A, B > 0\\
\dot x(t)=V_w(\theta(t))u(t) \qquad \forall t \in [0, t_f],\\
\dot{\theta}(t)=u(t), \qquad \forall t \in [0, t_f],\\
w(t)=h_{\varepsilon}[u](t),\\
u\in C^0[0,t_f], \qquad u(0)=u(t_f)=u_0,\\
u(t) \in U:= [-\varepsilon, \varepsilon],\qquad \forall t \in [0,t_f], \\
x(0)=x_0,\qquad x(t_f)=x_f,\\
\theta(0)=\theta_0, \qquad \theta(t_f)=\theta_0.
\end{cases}
\end{equation}
that is a generalization of the linear quadratic cost associated to the linearized dynamics, thus in some sense can be considered as a generalized energy. In order to compute the optimal solution we proceed as before: first we relax the problem considering controls $u \in L^{2}(0, t_f)$. Then exploiting the properties of $\Delta x$ as function of $\theta(t_1)$, since there is a unique switching dynamics and the running cost is nonnegative, we divide the integral in \eqref{quadratic} in two parts. The first, in the time interval $[0, t_1]$, the second one in $[t_1, t_f]$. Hence consider the~first problem
\begin{equation}\label{quadraticfirst}
\begin{cases}
\displaystyle\min_{u \in [-\varepsilon, \varepsilon]} J[u]=\min_{u \in [-\varepsilon, \varepsilon]} \int_0^{t_1} A u^2(s)+B \theta(s)^2\, ds\\
\dot{\theta}(t)=u(t), \qquad \forall t \in [0, t_1],\\
u\in U=[-\varepsilon,\varepsilon],\quad u\in L^2[0,t_1],\\
\theta(0)=\theta_0, \qquad \theta(t_1)=\theta_1.
\end{cases}
\end{equation}
where we don't consider the initial and final condition on the control $u$ because now it is $L^2(0, t_1)$.
Let us now focus on the case $\theta_1>\theta_0$, (the reverse inequality can be treated similarly) and apply the PMP considering the associated  Hamiltonian
\begin{equation*}
H(\theta, u, p)=Au^2+B\theta^2+pu
\end{equation*}
Then, through the necessary condition for optimality, we get $\forall t \in [0, t_1]$
\begin{equation*}
\begin{cases}
 p(t)=-2\theta_0\sqrt{AB} \exp\left(\frac{\sqrt{B}}{\sqrt{A}}t\right),\\ 
 u(t)=\frac{\sqrt{B}}{\sqrt{A}}\theta_0\exp\left(\frac{\sqrt{B}}{\sqrt{A}}t\right),
\end{cases}\quad  \text{if} \quad \theta_1>\theta_0>0
\end{equation*}
Now we have to consider the constraint  $|u(t)|\leq \varepsilon$. 
Since the control $u$ is positive we get the following condition on $t$ 
\begin{equation}\label{condizione1}
t\leq\sqrt{\frac{A}{B}}\log\Bigl(\sqrt{\frac{A}{B}}\frac{\varepsilon}{\theta_0}\Bigr).
\end{equation}
In \eqref{condizione1} if $\varepsilon$ is small compared to $A, B, \theta_0$ one gets that $\log\Bigl(\sqrt{\frac{A}{B}}\frac{\varepsilon}{\theta_0}\Bigr)<0$, therefore for $t>0$ 
we have 
\begin{equation}\label{controllo1caso}
u(t)=\varepsilon.
\end{equation}


Thus considering the condition $\theta(t_1)=\theta_1$ we have
\begin{equation*}
t_1=
 (\theta_1-\theta_0)/{\varepsilon} 
\end{equation*}
The second problem is completely analogous to the first one \eqref{quadraticfirst} but in the time interval $[t_1,t_f]$ and with reversed initial and final conditions on $\theta$. Therefore the solution can be computed in a similar way
\begin{equation*}
\begin{cases}
 p(t)=2\theta_1\sqrt{AB} \exp\left(-\frac{\sqrt{B}}{\sqrt{A}}t\right), \\  u(t)=-\frac{\sqrt{B}}{\sqrt{A}}\theta_1\exp\left(-\frac{\sqrt{B}}{\sqrt{A}}t\right), 
\end{cases} \quad  \text{if} \quad \theta_1>\theta_0>0
\end{equation*}
Imposing again the constraint $|u(t)|\leq\varepsilon$ and following the considerations made for the first problem we obtain a condition on $t$ which leads to the following control 
\begin{equation*}
u(t)=\begin{cases}
-\varepsilon \quad  t<-\sqrt{\frac{A}{B}}\log\Bigl(\sqrt{\frac{A}{B}}\frac{\varepsilon}{\theta_1}\Bigr),\\
-\frac{\sqrt{B}}{\sqrt{A}}\theta_1\exp\left(-\frac{\sqrt{B}}{\sqrt{A}}t\right) \quad  t\geq-\sqrt{\frac{A}{B}}\log\Bigl(\sqrt{\frac{A}{B}}\frac{\varepsilon}{\theta_1}\Bigr)
\end{cases} 
\end{equation*}
Consequently we compute the corresponding final time
\begin{equation*}
t_2=t_f-t_1=
\begin{cases}
\frac{\sqrt{A}}{\sqrt{B}}\log\left(\frac{\theta_1}{\theta_0}\right) & \text {if} \quad \theta_1>\varepsilon\sqrt{\frac{A}{B}}>\theta_0,\\
(\theta_1-\theta_0)/{\varepsilon}& \text {if} \quad\theta_1>\theta_0>\varepsilon\sqrt{\frac{A}{B}},
\end{cases} \end{equation*}
which gives us
$$
t_f=\begin{cases}
\frac{\theta_1-\theta_0}{\varepsilon}+\sqrt{\frac{A}{B}}\log\left(\frac{\theta_1}{\theta_0}\right)& \text {if} \quad \theta_1>\sqrt{\frac{A}{B}}\varepsilon>\theta_0,\\
2(\theta_1-\theta_0)/{\varepsilon}& \text {if} \quad\theta_1>\theta_0>\sqrt{\frac{A}{B}}\varepsilon.\\
\end{cases}
$$
Note that the solution of \eqref{quadratic} $(\theta(t), u(t))$ for $t \in [0, t_f]$ is not continuous in $t_1 \in (0, t_f)$, hence we can consider a continuous approximation of the control $u$ through piecewise linear functions as done for the minimum time problem. For that we have to reconsider the initial and final condition $u(0)=u_0=u(t_f)$. In this way we get $u_k, \theta_k \in C^0[0, t_{f_k}]$  and $t_{f_k}$ such that
\begin{equation*}
\lim_{k\to +\infty}u_k=u \ \text{in} \ L^2([0, t_f]), \qquad \lim_{k\to +\infty}t_{f_k}=t_f.
\end{equation*}
Note that the above convergence of $u_k$ in $L^2$ is also valid in $L^1$ due to the boundedness of $[0, t_f]$.
Moreover the following hold
\begin{Proposition}
	The trajectory $\theta_k$ and the cost $J_k=J[u_k]$ converge uniformly to $\theta$ and $J$ respectively, i.e. 
	\begin{equation}
	\begin{aligned}
	\lim_{k\to +\infty} \theta_k = \theta \ \text{in} \ L^{\infty}([0, t_f]),\\
	\lim_{k\to +\infty} J_k = J \ \text{in} \ L^{\infty}([0, t_f]).
	\end{aligned}
	\end{equation}	
\end{Proposition}
\begin{proof}
The proof of the uniform converge of $\theta_k$ to $\theta$ is analogous to one of Proposition \ref{convtraiett}. While for the convergence of $J_k$ 
\begin{equation*}
\begin{split}
\Vert J_k-J\Vert_\infty & \leq \int_0^{t_{f_k}}A\vert u_k^2(s)-u_(s)\vert \,ds\\
&+\int_0^{t_{f_k}} B\vert \theta_k^2(s)-\theta(s)^2\vert\,ds\\ & \leq A\Vert u_k+u\Vert_\infty \int_0^{t_{f_k}}\vert u_k(s)-u_(s)\vert \,ds \\ &+ B\Vert  \theta_k+\theta\Vert_\infty\int_0^{t_{f_k}}\vert \theta_k-\theta\vert \,ds.
\end{split}
\end{equation*}
Therefore we conclude for the $L^1$ convergence of $u_k$ to $u$ and the $L^\infty$ convergence of $\theta_k$ to $\theta$
\end{proof}
\subsection{Case $B=0$}
Let us now consider a particular case of the optimal control problem \eqref{quadratic} i.e we fix $B=0$. Also in this case we can split the problem into  two subproblems. Considering the first time interval the cost is
$$
J[u] =\int_0^{t_1} u(s)^2 \,ds.
$$
Notice that since $\dot\theta=u$ and using the Jensen inequality we have the following
\begin{equation}
J[u]=\int_0^{t_1} \dot\theta(s)^2\,ds\geq t_1\Bigl(\frac{1}{t_1}\int_0^{t_1} \dot\theta(s)\,ds\Bigr)^2=\frac{(\theta_1-\theta_0)^2}{t_1}.
\end{equation}
Moreover, since $u\in[-\varepsilon,\varepsilon]$, the equality holds only if the control $u$ is constant, either $-\varepsilon$ or $\varepsilon$. Therefore we conclude that in this case the problem is equivalent to the minimal time one with
\begin{equation}
u(t)=\begin{cases}
\varepsilon & \text{if} \quad \theta_1>\theta_0\\
-\varepsilon & \text{if}\quad \theta_0>\theta_1
\end{cases}
\end{equation}
and consequently $t_1=(\theta_1-\theta_0)/
\varepsilon.$
The same considerations are valid for the second time interval, therefore the entire problem is equivalent to the minimal time one treated in Section \ref{Sec:Minimum_time}. This means that if we constrain the control in the compact set $[-\varepsilon,\varepsilon]$ the optimal strategy for minimizing the energy is equivalent to the one which minimize the time to go from the initial to the final configuration.
\subsection{Case on $n$ switchings}
Recall that for the problem \eqref{Initial2}  and \eqref{quadraticfirst} we analyse the simple case with only a switching.
Let now suppose to perform $n$ switches to go from $x_0$ to $x_f$. Therefore we divide $\Delta x$ in $n$ sub-intervals  assuming that only one switching is required to move from the start point to the end point of each one. Thus $$J^n(u,t)=\sum_{i=1}^n\int_{0}^{t_i}\ell_0(\theta,u,s)\,ds + n $$ with $\ell_0(\theta,u,s)=1$ in the minimum time case and $\ell_0(\theta,u,s)=u(s)^2$ in the minimum energy one, is the total cost to minimize using $n$ switchings with $t_i$ the minimum time used to cross the subinterval $i$. Then the following hold
\begin{Proposition}
\label{n_min_time}
	\begin{equation*}
	\min_{n,u}J^n(u,t)\leq \min_u \left\{J^1(u,t)+1\right\},
	\end{equation*}
	where $J^1(u,t)$ is the cost to minimize using only one switching.
\end{Proposition}
\begin{proof}
The proof follows by the definition of minimum of a set.
\end{proof}
In computing the explicit solution both of the minimum time problem and the minimum quadratic cost one we have supposed that to reach $x_f$ from $x_0$ only one switching in the dynamics is required. Now, we will show numerically that our assumption on the number of switching is also the optimal strategy for both the studied problems. We fix $\Delta x=x_f-x_0=10$ and use the following parameters
\begin{itemize}
	\item $a=10,\ b=0.1$, $\xi=1,\ \eta=2$;
	\item $m=1,\ m_{11}= a\pi, \ m_{22}=b\pi$;
	\item $\theta_0=\pi/6,\ 	A=1,\ B=0,\ n=30$.
\end{itemize}  
We divide $\Delta x$ into $n$ subintervals in each of them we compute the optimal time (and the control $L^2$ norm) and then we sum them up obtaining the total time (and the total quadratic cost) needed to cross the interval $[x_0,x_f]$. Finally we add to the total cost the number of switchings obtaining $J^n(u,t)$. Plotting each cost in function of the number $n$ of switchings that we perform, we get that the greater the number of switchings, the greater the time to reach $x_f$ is. Same conclusion also holds for the energy. See Fig. \ref{funzcrescente}. 
\begin{figure}[H]
	\centering%
		\subfigure[\label{fig1}]%
		{\includegraphics[scale=0.289]{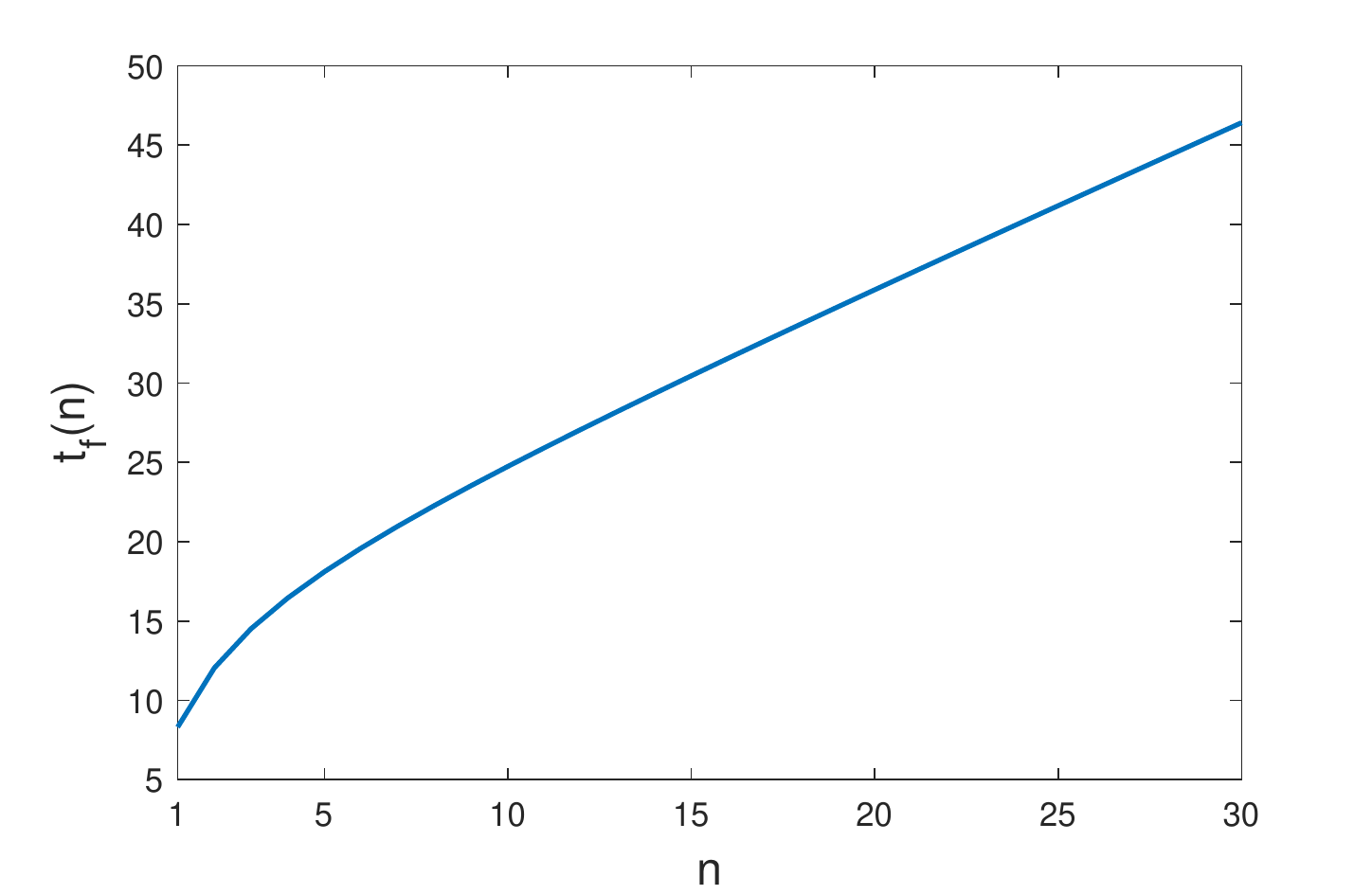}}
		\subfigure[\label{fig2}]%
		{\includegraphics[scale=0.289]{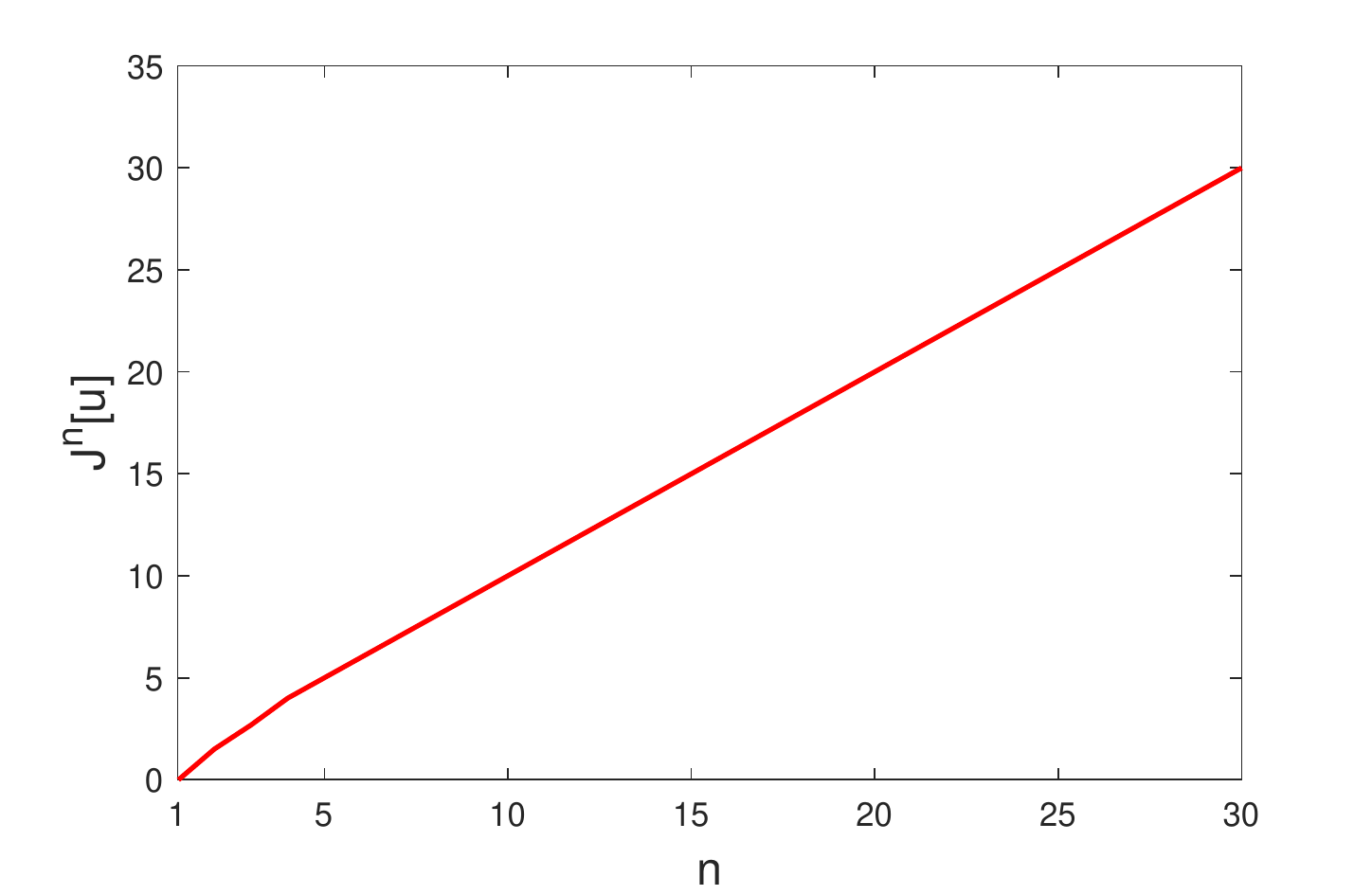}}
		\caption{\label{funzcrescente} Plot of the time (a) and of the cost (b) to realize the displacement $\Delta_x$ as function of the number $n$ of switchings.}
	\end{figure}

\section{CONCLUSIONS}
In this paper we explicitly derive the solution of two different optimal control problems (the minimum time one and the linear quadratic one) related to the motion of the scallop swimmer with one switching in the dynamics. We approximate the possible discontinuous solution with a suitable sequence of periodic continuous functions. Comparing the two problem we find that the optimal strategy for minimizing the $L^2$ norm of the control is equivalent to the one which minimize the time to reach the final configuration. The more general case with $n$ switching in the dynamics is then treated and numerical simulations showing the relations between $n$ and a generalized cost function $J^n$ are provided.
In future research, inspired by the numerical results, we plan to give an analytic proof about the above relations and study other optimal control problems.

%
\section*{ACKNOWLEDGMENT}
\small
The authors thank Fabio Bagagiolo and Antonio Marigonda for many enlighting and helpful discussions.

\end{document}